\documentclass[11pt,reqno]{amsart}
\usepackage[text={160mm, 235mm},centering]{geometry}
\usepackage{float}
\usepackage{appendix}
\usepackage{graphicx}
\usepackage{chngpage}
\usepackage{multicol}
\usepackage{subcaption}
\usepackage{setspace}
\usepackage{multirow}
\usepackage[colorlinks=true, citecolor=blue, urlcolor=blue, linkcolor=blue]{hyperref}
\usepackage[all,cmtip]{xy}
\usepackage{amssymb}
\usepackage{tikz-cd}
\usepackage{amsmath}
\usepackage{relsize}
\usepackage{longtable}
\usepackage{pstricks}
\usepackage{color}
\usepackage{ mathrsfs }
\usepackage[all]{xy}
\usepackage{mathabx,epsfig}

\theoremstyle{plain}
\newtheorem{thm}{Theorem}[section]
\newtheorem{theorem}[thm]{Theorem}

\newtheorem{lemma}[thm]{Lemma}
\newtheorem{corollary}[thm]{Corollary}

\newtheorem*{Assumption}{Assumption A}
\newtheorem*{Problem}{Problem}
\newtheorem{question}[thm]{Question}
\theoremstyle{definition}
\newtheorem{remark}[thm]{Remark}

\newtheorem{definition}[thm]{Definition}

\numberwithin{equation}{section}

\newcommand{\romea}{\uppercase\expandafter{\romannumeral1}}
\newcommand{\romeb}{\uppercase\expandafter{\romannumeral2}}

\newcommand{\Aut}{{\rm Aut}}

\newcommand{\Diag}{{\rm diag}}

\newcommand{\PGL}{{\rm PGL}}

\newcommand{\C}{{\mathbb C}}

\renewcommand{\P}{{\mathbb P}}
\newcommand{\Q}{{\mathbb Q}}

\newcommand{\Z}{{\mathbb Z}}

\makeatletter
\@namedef{subjclassname@2020}{%
	\textup{2020} Mathematics Subject Classification}
\makeatother

\begin{document}

	\title{Nonrational varieties with unirational parametrizations of coprime degrees}

\author{Song Yang, Xun Yu \and Zigang Zhu}

\address{Center for Applied Mathematics and KL-AAGDM, Tianjin University, Weijin Road 92, Tianjin 300072, P.R. China}

\email{syangmath@tju.edu.cn, xunyu@tju.edu.cn, zhzg0313@tju.edu.cn}


\maketitle

\begin{abstract}
	We show that there exists a $2$-dimensional family of smooth cubic threefolds admitting unirational parametrizations of coprime degrees. This together with Clemens--Griffiths' work solves the long standing open problem whether there exists a nonrational variety with unirational parametrizations of coprime degrees. Our proof uses a new approach, called the Noether--Cremona method, for determining the rationality of quotients of hypersurfaces.
\end{abstract}
\setcounter{tocdepth}{1}
\section{Introduction}
 Varieties are called rational if they are birational to projective space. It is a fundamental problem in algebraic geometry to determine which varieties are rational. Rational varieties must be unirational, that is, dominant by a rational map from projective space. The classical L\"{u}roth problem asks if every unirational variety is rational. The first counter-examples appeared in 1970s. Clemens--Griffiths \cite{CG72} proved that every smooth cubic threefold is irrational by intermediate Jacobian. Iskovskikh--Manin \cite{IM71} showed that every smooth quartic threefold is irrational by birational rigidity. Artin--Mumford \cite{AM72} proved that some quartic double solids are unirational but not rational by torsion in $H^3$. Since then, more examples of nonrational unirational varieties are found (see e.g. \cite{BCTSSD85}, \cite{Puk87}, \cite{Kol95}, \cite{Puk98}, \cite{deF13}, \cite{Voi14},  \cite{COV15}, \cite{CTP16}, \cite{Tot16}, \cite{HPT18},
 \cite{Sch19a}, \cite{Sch19b}, \cite{Sch21}, \cite{NO22}). However, the following long standing problem is still open (see e.g. \cite[\S 1.2.1]{HMP98}, \cite[Remark 3.9]{CT17}, \cite[Question 7.2.2]{AB17} and \cite[Chapter 6, Remark 5.17]{Huy23}).
	
	\begin{Problem}
		Does there exist a nonrational variety with unirational parametrizations of coprime degrees?
		In particular, is there a smooth cubic threefold admitting unirational parametrizations of coprime degrees?
	\end{Problem}
Here, an $n$-dimensional variety $X$ is said to admit a \textit{unirational parameterization of degree $d$} if there exists a degree $d$ dominant rational map from the projective space $\mathbb{P}^n$ to $X$.
	It is known that every smooth cubic hypersurface in dimension at least three admits a unirational parametrization of degree $2$ (see \cite[Appendix B]{CG72}). The moduli space of smooth cubic fourfolds contains some divisors whose members admit unirational parametrizations of odd degree (see e.g. \cite{HT01}, \cite{Has16}). Existence of nonrational ones among such members is unclear since the rationality of cubic fourfolds is still not well-understood in general. 
	
	The main result of this paper is as follows.
	
	\begin{theorem}\label{thm:main}
		There exists a $2$-dimensional family of smooth cubic threefolds admitting unirational parametrizations of coprime degrees.
	\end{theorem}
Actually these degrees are the smallest coprime degrees 2 and 3 (Theorem \ref{thm:maindetailed}).
	This together with Clemens--Griffiths' work \cite{CG72} solves the above-mentioned open problem. 
	Nothing is known about the rationality of smooth cubic hypersurfaces of odd dimension $\geq 5$. For such cubic hypersurfaces, based on Theorem \ref{thm:main} and \cite[Proposition 3.1]{CT17}, we obtain the following
	
	\begin{theorem}\label{thm:main-cor}
Let $n\geq5$ be an odd integer. Then every smooth cubic hypersurface $X \subset \mathbb{P}^{n+1}$ of dimension $n$ whose equation is given by a form $\sum_i F_i$, where each $F_i$ is in separated variables and involves at most three variables, admits unirational parametrizations of coprime degrees.
	\end{theorem}	
	Next, we briefly explain the idea of the proof of Theorem \ref{thm:main}. We develop a new approach, called the {\it Noether--Cremona method}, for determining the rationality of hypersurfaces and their quotients by finite groups.  
	This method rises from our study \cite{YYZ25} on the rationality of quotients of cubic threefolds by finite groups classified in \cite{WY20}, 
	which is inspired by the classical Noether's problem (see \cite{Pro10} for a survey). Roughly speaking, suppose a smooth hypersurface $X$ in the complex projective space $\P^{n+1}$ is preserved by a finite subgroup $G$ in the automorphism group ${\rm Aut}(\P^{n+1})$ of $\P^{n+1}$. By choosing a birational map (if it exists), called the {\it Noether--Cremona transformation}, $\P^{n+1}/G \dashrightarrow \P^{n+1}$, we transfer the quotient variety $X/G$ birationally to a hypersurface $X_{NC}$ (Theorem \ref{thm:bir}). The degree of $X_{NC}$ might be greater than that of $X$. A crucial point is that by repeating this process with suitable Noether--Cremona transformations, we often can find some hypersurface birational to $X/G$ of degree close to that of $X$, which enables us to determine the rationality/irrationality of $X/G$. More precisely, we find a $2$-dimensional family of smooth cubic threefolds $X$ preserved by $G$ isomorphic to the elementary abelian group of order $9$. By applying Noether--Cremona method, we show that $X/G$ is birational to $X'$ lying in another $2$-dimensional family of smooth cubic threefolds (Theorem \ref{thm:C3C3}). On the other hand, applying Noether--Cremona method again, we show that there exists a (normal) subgroup $G_1$ in $G$ of order $3$ such that $X/G_1$ is rational (Theorem \ref{thm:C3cubic3}). In this way, we conclude that the smooth cubics $X'$ admit a unirational parametrization of degree $3$ (Theorem \ref{thm:maindetailed}). The Noether--Cremona method proves particularly effective for quotients of cubic hypersurfaces in \cite{YYZ25}. 
	
It seems that a common feature of most (if not all) previously known approaches to the rationality problem is that each is effective for establishing either rationality or irrationality, but not both. However, it is worth noting that the Noether--Cremona method can be used to demonstrate both rationality and irrationality. Although we focus on complex varieties in this paper, thanks to the nature of Noether--Cremona method (see Remark \ref{rem:coeff}), the Noether--Cremona method works well for an arbitrary ground field (see Remarks \ref{rmk:basefield}, \ref{rmk:posi}).
	
	We conclude the introduction by posing some open questions. If a variety admits unirational parametrizations of coprime degrees, 
	then it has a Chow-theoretic decomposition of the diagonal; equivalently, it is universally ${\rm CH}_{0}$-trivial (see \cite{ACTP17}). There is a nonempty countable union of proper subvarieties of codimension $\leq 3$ in the moduli space of smooth cubic threefolds parametrizing universally ${\rm CH}_0$-trivial cubic threefolds (\cite{Voi17}). These subvarieties contain the smooth cubic threefolds in Theorem \ref{thm:maindetailed}.

	\begin{question}
Is there a smooth cubic threefold, isomorphic to none of the smooth cubic threefolds in Theorem \ref{thm:maindetailed}, that admits unirational parametrizations of coprime degrees?
	\end{question}
Thanks to the breakthrough work of Voisin \cite{Voi14}, all recently known methods for proving stable irrationality rely on the decomposition of the diagonal. Very recently, Engel--de Gaay-Fortman--Schreieder \cite[Corollary 1.4]{EdeGFS25} proved that very general cubic threefolds do not admit a decomposition of the
diagonal, and hence not stably rational. The relation between stable rationality and the existence of unirational parametrizations of coprime degrees is unclear (see \cite[Page 1620]{Voi17}).
	\begin{question}
		Is there a smooth complex projective variety admitting unirational parametrizations of coprime degrees which is not stably rational?
	\end{question}
It is known that (stably)
rationality is preserved under specializations in smooth families (\cite{NS19}, \cite{KT19}). On the other hand, analogous result for unirationality is unknown.
	\begin{question}
	Is the existence of unirational parametrizations of coprime degrees preserved under specializations in smooth families?
	\end{question}

	\subsection*{Acknowledgements}
	We would like to thank Professor Keiji Oguiso for valuable conversations and comments.
	This work is partially supported by the National Natural Science Foundation of
	China (No. 12171351, No. 12071337).

\section{Noether--Cremona method}
 In this section, we will introduce the Noether--Cremona method and apply it to study the rationality of quotients of hypersurfaces. 
	As illustrations, we investigate the rationality of quotients of smooth cubic hypersurfaces in dimension $\geq 3$ by two different actions of the order $3$ cyclic group (Theorems \ref{thm:Ex1} and \ref{thm:Ex3}).
\subsection{General set-up}
Let $G$ be a finite subgroup in the automorphism group $\Aut(\P^{n+1})$ of the complex projective space $\P^{n+1}$ and let $X\subset \P^{n+1}$ be an irreducible hypersurface of degree $d\geq2$, which is preserved by $G$ (i.e. $g(X)=X$ for any $g\in G$). Then  the quotient variety $X/G$ is a closed subvariety of $\P^{n+1}/G$. Let $H\subset\P^{n+1}$ be a hyperplane preserved by $G$. 
\begin{definition}
	Let $G\subset \Aut(\P^{n+1})$ be a finite subgroup. If $\Phi: \P^{n+1}/G\dashrightarrow \P^{n+1}$ is a birational map, then we call $\Phi$ a {\it Noether--Cremona transformation}.
\end{definition}
Note that the rationality of $\mathbb{P}^{n+1}/G$  is known as the {\it Noether's problem} (\cite{Noe13}), a classical question with a long history (see e.g. \cite{Pro10}). In general, it is still an open problem.
We make the following \begin{Assumption}\label{assum}
	$\P^{n+1}/G$ is rational. 
\end{Assumption}
Under this assumption, by choosing various Noether--Cremona transformations $\Phi: \P^{n+1}/G\dashrightarrow \P^{n+1}$, we transfer $X/G$ to hypersurfaces $Y\subset\P^{n+1}$ often birational to $X/G$. By studying rationality of such $Y$, we determine the rationality of $X/G$.
The idea of this method is illustrated in the following diagram
\begin{center}
	\begin{tikzcd}
		\mathbb{P}^{n+1} \arrow[r] \arrow["G"', loop, distance=1.3em, in=190, out=170]  & \mathbb{P}^{n+1}/G \arrow[r, "\Phi", dashed] & \mathbb{P}^{n+1}  \\
		X \arrow[u, hook] \arrow[r] \arrow["G"', loop, distance=1.3em, in=195, out=165] & X/G \arrow[u, hook] \arrow[r, dashed]        & \, Y \, . \arrow[u, hook]
	\end{tikzcd}
\end{center}

\subsection{Special case: $G$ being abelian}\label{ss:Gabe}
We fix some notation and conventions which will be used throughout this subsection. 
	\begin{enumerate}
		\item[]
		\begin{tabular}{rl} 
			$G$ & an abelian subgroup in ${\rm PGL}(n+2,\C)$ generated by $[A_1],\dots, [A_k]$,\\ & where $A_i={\rm diag}(\lambda_{i1},\dots,\lambda_{i(n+1)},1)$ ($1\le i\le k$) and all $\lambda_{ij}$ are roots of unity;\\
			
			$\tilde{G}$ & the abelian subgroup in ${\rm GL}(n+2,\C)$ generated by $A_1,\dots, A_k$;\\
			
			$F$ & an irreducible homogeneous polynomial in variables $x_1,\dots,x_{n+2}$ of degree $d\ge 2$\\
			&satisfying $F(\lambda_{i1}x_1,\dots,\lambda_{i(n+1)}x_{n+1},x_{n+2})=F$ for all $i$;\\
			
			$X$ & the hypersurface  in $\P^{n+1}$ defined by $F$; \\
			
			$H$ &  the hyperplane $\{ x_{n+2}=0\}$ in $\P^{n+1}$.\\
		\end{tabular} 
	\end{enumerate} 
Recall that the group $G$ acts on $\P^{n+1}$ via $$([A_i], (x_1:\cdots:x_{n+1}:x_{n+2})) \mapsto (\lambda_{i1}x_1:\cdots:\lambda_{i(n+1)}x_{n+1}:x_{n+2}).$$ Thus we may view $G$ as a finite subgroup in ${\rm Aut}(\P^{n+1})$ which preserves $X$. 

In this case, {\bf Assumption A} is satisfied by Fischer \cite{Fis15}.
Before we introduce the Noether--Cremona method, we need some preparation. Note that $\tilde{G}$ acts on $\C[x_1,\dots,x_{n+2}]$ via $$A_i(f)= f(\lambda_{i1}x_1,\dots,\lambda_{i(n+1)}x_{n+1},x_{n+2}),\;\;\; f\in\C[x_1,\dots,x_{n+2}],$$
which induces an action of $\tilde{G}$ on the subring  $\C[x_1,\dots,x_{n+1}]$. 
We call $\mathfrak{m}\in\C(x_1,\dots,x_{n+1})$ a {\it rational monomial}
if it can be written as $\mathfrak{m}=x_1^{\alpha_1}\cdots x_{n+1}^{\alpha_{n+1}}$, where $\alpha_i\in \Z$.
By \cite[Theorem 1]{Cha69}, there exists a set of rational monomials 
\begin{equation}\label{eq:ui}
	u_i=x_1^{a_{i1}}\cdots x_{n+1}^{a_{i(n+1)}} \;\;(1\le i\le n+1)
\end{equation}
such that the invariant field
\begin{equation}\label{eq:invfield}
	\C(x_1,\dots,x_{n+1})^{\tilde{G}}=\C(u_1,\dots,u_{n+1}).
\end{equation} 
\begin{lemma}\label{lem:mono}
	Let $\mathfrak{m}\in\C(x_1,\dots,x_{n+1})^{\tilde{G}}$ be a rational monomial. Then $\mathfrak{m}$ can be written as a rational monomial in variables $u_1,\dots,u_{n+1}$.
\end{lemma}
\begin{proof}
	The $(n+1)\times (n+1)$ matrix $P=(a_{ij})$  has nonzero determinant since $u_1,\dots,u_{n+1}$ are algebraically independent by equations \eqref{eq:ui} and \eqref{eq:invfield}. From this, we infer that there exist positive integers $d_i$ such that  $x_i^{d_i}$ can be written as rational monomials in  variables $u_i$. Then there exists a positive integer $N$ such that $\mathfrak{m}^N$ can be written as  a rational monomial in variables $u_i$, which concludes the lemma by the equality \eqref{eq:invfield} and $\C[u_1,\dots,u_{n+1}]$ being a UFD.
\end{proof}
\begin{lemma}\label{lem:local}
 The localization $\C[u_1,\dots,u_{n+1}]_{u_1\cdots u_{n+1}}$ is equal to $\left(\C[x_1,\dots,x_{n+1}]_{x_1\cdots x_{n+1}}\right)^{\tilde{G}}$.
\end{lemma}
\begin{proof}
	Every element in the invariant ring  $\left(\C[x_1,\dots,x_{n+1}]_{x_1\cdots x_{n+1}}\right)^{\tilde{G}}$ is a finite sum of rational monomials in variables $x_i$ with nonzero coefficients. Thus the lemma follows from Lemma \ref{lem:mono}.
\end{proof}
\subsection*{Noether--Cremona method} Our goal is to find hypersurfaces in $\P^{n+1}$ birational to $X/G$ using Noether--Cremona transformations. For this purpose, we proceed as follows.
\begin{enumerate}
	\item [1.] We choose a set of rational monomials $$u_i=x_1^{a_{i1}}\cdots x_{n+1}^{a_{i(n+1)}} \;\;(1\le i\le n+1)$$
	 such that $\C(x_1,\dots,x_{n+1})^{\tilde{G}}=\C(u_1,\dots,u_{n+1})$. Such choice of $u_i$ corresponds to a birational map $\Phi_U: U/G\dashrightarrow V$ given by $\Phi_U^*(u_i)=x_1^{a_{i1}}\cdots x_{n+1}^{a_{i(n+1)}}$, where local charts
	$$U=\{x_{n+2}=1\}= \P_{(x_1:\cdots:x_{n+2})}^{n+1}\setminus H {\;\;\rm and \;\;} V=\{u_{n+2}=1\}\subset \P_{(u_1:\cdots:u_{n+2})}^{n+1}.$$
	 This naturally gives a Noether--Cremona transformation 
	 \begin{equation}\label{eq:Phi}
	 	\Phi: \P_{(x_1:\cdots:x_{n+2})}^{n+1}/G \dashrightarrow \P_{(u_1:\cdots:u_{n+2})}^{n+1}
	 \end{equation}
	 such that $\Phi|_{U/G}=\Phi_U$.
	\item [2.] Let $f(x_1,\dots,x_{n+1}):=F(x_1,\dots,x_{n+1},1)$. Since $f\in \C[x_1,\dots,x_{n+1}]^{\tilde{G}}$, by Lemma \ref{lem:local}, there exist $p\in \C[u_1,\dots,u_{n+1}]$ and a monomial $q=u_1^{a_1'}\cdots u_{n+1}^{a_{n+1}'}$ (all $a_i'\ge 0$) such that
	$$f=\frac{p(u_1,\dots,u_{n+1})}{q(u_1,\dots,u_{n+1})}\, \text{ and }\, \gcd(p,q)=1.$$
	
	\item [3.] Replacing $u_i$ by $x_i/x_{n+2}$, we get a rational function $$f_1:=p(x_1/x_{n+2},\dots,x_{n+1}/x_{n+2})\in \C(x_1,\dots,x_{n+2}).$$ Then there exist $p_1\in \C[x_1,\dots,x_{n+2}]$ and a monomial $q_1=x_{n+2}^{d'}$ such that
	$$f_1=\frac{p_1}{q_1}\, \text{ and }\, \gcd(p_1,q_1)=1.$$
\end{enumerate}

	Let $F_{NC}:=p_1$, $X_{NC}:=\{F_{NC}=0\}\subset \P^{n+1}$ and $d_{NC}:=d'$. Note that $\deg(p_1)=\deg(q_1)=d_{NC}$. We call $F_{NC}$ (resp. $X_{NC}$) a {\it Noether--Cremona polynomial} (resp. {\it Noether--Cremona hypersurface}) of the pair $(F, G)$ (resp. the quotient $X/G$) of degree $d_{NC}$.
\begin{theorem}\label{thm:bir}
Every Noether--Cremona hypersurface $X_{NC}$ of $X/G$ is birational to $X/G$.
\end{theorem}
\begin{proof}
Let $U_1:=U\cap\{x_1\cdots x_{n+1}\neq0\}$ and $V_1:=V\cap\{u_1\cdots u_{n+1}\neq0\}$. Then from the Noether--Cremona transformation $\Phi$ in \eqref{eq:Phi}, we have the following commutative diagram
\begin{center}
	\begin{tikzcd}
		U_1 \arrow[d, "\pi"'] \arrow[rd, "\varphi"] &     \\
		U_1/G \arrow[r, "\bar{\varphi}"]            & V_1\,,
	\end{tikzcd}
\end{center}
where $\pi$ is the quotient map and $\varphi=\Phi|_{U_1}$. Since $V_1={\rm Spec}\,\C[u_1,\dots,u_{n+1}]_{u_1\cdots u_{n+1}}$ and $U_1/G={\rm Spec}\,\left(\C[x_1,\dots,x_{n+1}]_{x_1\cdots x_{n+1}}\right)^{\tilde{G}}$, by Lemma \ref{lem:local}, $\bar{\varphi}$ is a biregular morphism.
Let $X_1:=X\cap U_1$ and $Y_1:=\{p=0\}\cap V_1$. Since $X$ is irreducible and of degree $d\geq 2$, $X_1$ is nonempty. The quotient $X_1/G$ is isomorphic to $Y_1$ under $\bar{\varphi}$. Since $X_1/G$ and $Y_1$ are dense open subsets of $X/G$ and $X_{NC}$ respectively, we conclude that $X_{NC}$ is birational to $X/G$.
\end{proof}
\begin{remark}
	If $f$ is a polynomial in variables $u_i$ (equivalently, $q=1$) of degree less than $d$, then the Noether--Cremona hypersurface $X_{NC}$ is of degree less than $d$. This sometimes can be achieved by suitable choice of local chart and rational monomials $u_i$ (see the proof of Lemma \ref{lem:HtoHtil}).
\end{remark}
\begin{remark}\label{rem:coeff}
	The numbers of the monomials with nonzero coefficients in $F$ and $F_{NC}$ are equal.
	Moreover, the sets of such coefficients of $F$ and $F_{NC}$ are the same.
\end{remark}
\begin{remark}\label{rmk:basefield}
	Under the conventions in this subsection, 
	let $e$ be the exponent of the finite abelian group $G$.
	If we replace the complex number field $\C$ by any field $K$ containing an $e$-th primitive root of unity, then the Noether--Cremona method works and Theorem \ref{thm:bir} holds. In fact, the condition for the field $K$ is needed in \cite{Fis15} and \cite{Cha69}.
\end{remark}
\subsection{The first example}\label{ss:Ex}
We adopt the notation in Subsection \ref{ss:Gabe}. Let 
\begin{equation}\label{eq:G}
 G=\left\langle [A_1]\right\rangle \subset\PGL(n+2,\C),
\end{equation} where $A_1=\Diag(\xi_3,\xi_3^2,1,\dots,1)$, $n\geq3$ and $\xi_3$ is a third primitive root of unity. Let $X\subset \P^{n+1}$ be an irreducible cubic hypersurface defined by 
\begin{equation}\label{eq:ex1F}
F=t_1x_1^3+t_2x_2^3+lx_1x_2+h(x_3,\dots,x_{n+2}),
\end{equation} 
where $l(x_3,\dots,x_{n+2})=\sum_{3\leq i\leq {n+2}}t_ix_i$, $t_1,t_2\in\C^*:=\C \,\setminus \{0\}$, $t_i\in\C$ ($i\geq 3$) and $h$ is homogeneous of degree 3 with $\gcd(h,x_{n+2})=1$. Recall that there is a decomposition $$\C[x_1,\dots,x_{n+2}]=\bigoplus_{i\geq0}\C[x_1,\dots,x_{n+2}]_i,$$
where $\C[x_1,\dots,x_{n+2}]_i$ are the spaces of homogenous polynomials of degree $i$.
Note that the space of cubic forms invariant by ${\tilde{G}}=\left\langle A_1\right\rangle $ is 
$$\C[x_1,\dots,x_{n+2}]^{\tilde{G}}_3={\rm span}_\C\{x_1^3,x_2^3,x_1x_2x_3,\dots,x_1 x_2 x_{n+2}\}\oplus\C[x_3,\dots,x_{n+2}]_3.$$
Then the conventions in Subsection \ref{ss:Gabe} are satisfied. Next, we follow the procedure of Noether--Cremona method.

Step 1. For any monomial $\mathfrak{m}=x_1^{\alpha_1}\cdots x_{n+1}^{\alpha_{n+1}}$, $\alpha_i\geq0$, we have $A_1(\mathfrak{m})=\xi_3^{\alpha_1+2\alpha_2}\cdot\mathfrak{m}$. Thus $\mathfrak{m}$ is $\tilde{G}$-invariant if and only if $\alpha_1+2\alpha_2\equiv0 \;({\rm mod} \;3)$. From this, we infer that 
the invariant ring $\C[x_1,\dots,x_{n+1}]^{\tilde{G}}=\C[x_1^3, x_2^3, x_1x_2, x_3, \dots, x_{n+1}]$.
 We choose rational monomials
 $$u_1=x_1x_2,\;u_2=\frac{x_2^2}{x_1},\; u_i=x_i, \; i\in\{3,\dots,n+1\}.$$
 Since $x_1^3=u_1^2/u_2$, $x_2^3=u_1u_2$, $x_1x_2=u_1$ and $x_i=u_i$ for $i\geq3$, we have
 $\C(x_1,\dots,x_{n+1})^{\tilde{G}}=\C(u_1,\dots,u_{n+1}).$
 
 Step 2.  We have
 $$	f=t_1x_1^3+t_2x_2^3+l(x_1,\dots,x_{n+1},1)x_1x_2+h(x_3,\dots,x_{n+1},1).$$
 By $$f=t_1\frac{u_1^2}{u_2}+t_2u_1u_2+l(u_3,\dots,u_{n+1},1)u_1+h(u_3,\dots,u_{n+1},1),$$
 we have
 $$p=t_1u_1^2+t_2u_1u_2^2+l(u_3,\dots,u_{n+1},1)u_1u_2+u_2h(u_3,\dots,u_{n+1},1)\; {\rm and }\; q=u_2.$$
 
  Step 3.  We have
  $$f_1=t_1(\frac{x_1}{x_{n+2}})^2+t_2\frac{x_1}{x_{n+2}}(\frac{x_2}{x_{n+2}})^2+l(\frac{x_3}{x_{n+2}},\dots,\frac{x_{n+1}}{x_{n+2}},1)\frac{x_1x_2}{x_{n+2}^2}+\frac{x_2}{x_{n+2}}h(\frac{x_3}{x_{n+2}},\dots,\frac{x_{n+1}}{x_{n+2}},1).$$
  Then 
  $$p_1=t_1x_1^2x_{n+2}^2+t_2x_1x_2^2x_{n+2}+l(x_3,\dots,x_{n+2})x_1x_2x_{n+2}+x_2h(x_3,\dots,x_{n+2})$$
  and $q_1=x_{n+2}^4.$
  Thus the Noether--Cremona hypersurface $X_{NC}\subset \P^{n+1}$ is a quartic hypersurface defined by $F_{NC}=p_1$.

In order to get a cubic hypersurface birational to $X/G$, we apply the Noether--Cremona method to the quotient of $X_{NC}$ by the trivial group. Let $X':=X_{NC}$, $F':=F_{NC}$, $G':=\langle [A_1'] \rangle$, $\tilde{G'}:=\langle A_1' \rangle$, $H':=\{x_2=0\}$, where $A_1'=I_{n+2}$ (the identity matrix of rank $n+2$). Clearly the conventions in Subsection \ref{ss:Gabe} hold for the quintuple $(G',\tilde{G'}, X',F',H')$.
\begin{lemma}\label{lem:HtoHtil}
		The cubic hypersurface $X'_{NC}\in\P^{n+1}$ defined by 
	\begin{equation}\label{eq:ex1Fpnc}
	F'_{NC}=t_1 x_1^2 x_2+	t_2 x_1 x_2^2 +lx_1x_2+h(x_3,\dots,x_{n+2})
  \end{equation}
	is a Noether--Cremona hypersurface of $X'/G'$.
\end{lemma}
\begin{proof}
We choose the affine chart $U' := \{x_2 = 1\}$ and take
	$u_1=x_1x_{n+2}$, $u_i=x_i$ for $i\in\{3,\dots,n+2\}$.
	Clearly  $\C(x_1,x_3,\dots,x_{n+2})^{\tilde{G'}}=\C(u_1,u_3,\dots,u_{n+2}).$ Then we have
	 \begin{align*}
	f&=t_1 x_1^2 x_{n+2}^2+	t_2 x_1 x_{n+2} +lx_1x_{n+2}+h(x_3,\dots,x_{n+2})\\
&=t_1u_1^2+t_2u_1+l(u_3,\dots,u_{n+2})u_1+h(u_3,\dots,u_{n+2}).
	\end{align*}
	Homogenizing $p=f$ by replacing $u_i$ by $x_i/x_{2}$, we obtain $F'_{NC}$.
\end{proof}
\begin{lemma}\label{lem:smtosm}
	The cubic form $F$ in \eqref{eq:ex1F} is smooth if and only if $F'_{NC}$ in \eqref{eq:ex1Fpnc} is smooth.
\end{lemma}
\begin{proof}
Suppose $X$ has a singular point $P = (x_1 : x_2 : \dots : x_{n+2})$. If $x_1 = 0$ or $x_2 = 0$, then it is clear that $(0 : 0 : x_3 : \dots : x_{n+2})$ is a singular point of $X'_{NC}$. If $x_1 x_2 \ne 0$, one can check that the point $\left(\frac{x_1^2}{x_2} : \frac{x_2^2}{x_1} : x_3 : \dots : x_{n+2}\right)$ is singular on $X'_{NC}$. Conversely, if $P = (x_1 : x_2 : \dots : x_{n+2})$ is a singular point of $X'_{NC}$, then $X$ is singular at $(0 : 0 : x_3 : \dots : x_{n+2})$ if $x_1 x_2 = 0$, and at $\left((\frac{t_1}{t_2})^{1/3} x_1 : (\frac{t_2}{t_1})^{1/3} x_2 : x_3 : \dots : x_{n+2}\right)$ if $x_1 x_2 \ne 0$.
\end{proof}
To summarize, by Theorem \ref{thm:bir}, Lemmas \ref{lem:HtoHtil}, \ref{lem:smtosm} and \cite[Theorem 13.12]{CG72}, we have the following
\begin{theorem}\label{thm:Ex1}
Let $X$ be a smooth cubic hypersurface in $\mathbb{P}^{n+1}$ defined by $F$ as in \eqref{eq:ex1F}. Let $G$ be as in \eqref{eq:G}. Then the quotient variety $X/G$ is birational to the smooth cubic hypersurface in $\P^{n+1}$ defined by $F'_{NC}$ as in \eqref{eq:ex1Fpnc}. In particular, if $n=3$, then $X/G$ is irrational.
\end{theorem}
\subsection{The second example}\label{ss:Ex3}
Again, we adopt the notation in Subsection \ref{ss:Gabe}. 
\begin{theorem}\label{thm:Ex3}
Let $X$ be a smooth cubic hypersurface in $\mathbb{P}^{n+1}$ defined by 
\begin{equation*}
F=t_1x_1^2x_2+t_2x_2^2x_1+t_3x_3^2x_4+t_4x_4^2x_3+l_1x_1x_3+l_2x_1x_4+l_3x_2x_3+l_4x_2x_4+h(x_5,\dots,x_{n+2}),
\end{equation*} 
where $n\ge 3$, $l_i(x_5,\dots,x_{n+2})=\sum_{5\leq j\leq {n+2}}t_{ij}x_j$, $t_1,t_2,t_3,t_4\in\C^*$, $t_{ij}\in\C$ and $h$ is homogeneous of degree $3$ with $\gcd(h,x_{n+2})=1$. Let $G=\left\langle [A_1]\right\rangle \subset\PGL(n+2,\C)$, where $A_1=\Diag(\xi_3,\xi_3,\xi_3^2,\xi_3^2,1,\dots,1)$. Then $X/G$ is rational.\end{theorem}

\begin{proof}
Similar to the proof of Theorem \ref{thm:Ex1}, we apply Noether--Cremona method for $X/G$. We take the affine chart $U=\{x_{n+2}=1\}$ and choose $\tilde{G}$-invariant monomials
$$u_1=x_2 x_3, \,u_2=x_2 x_4,\, u_3=x_1 x_3, \, u_4=x_3^2 x_4 \;{\rm and} \;u_i=x_i, \; i\in\{5,\dots,n+1\}.$$
By computing monomials in the spaces $\C[x_1,\dots,x_{n+1}]_i^{\tilde{G}}$ ($1\le i\le 3$), we infer that $$\C[x_1,\dots,x_{n+1}]^{\tilde{G}}=\C[x_1 x_3, x_1 x_4,x_2 x_3, x_2 x_4,x_1^3, x_1^2 x_2, x_1 x_2^2, x_2^3, x_3^3, 
x_3^2 x_4, x_3 x_4^2,x_4^3,x_5,\dots,x_{n+1}].$$ Since $x_1x_4=u_1^{-1}u_2u_3$, $x_1^3=u_1^{-1}u_2u_3^3u_4^{-1}$, $x_1^2 x_2=u_2u_3^2u_4^{-1}$, $x_1 x_2^2=u_1u_2u_3u_4^{-1}$, $x_2^3=u_1^2u_2u_4^{-1}$, $x_3^3=u_1u_2^{-1}u_4$, $x_3 x_4^2=u_1^{-1}u_2u_4$ and  $x_4^3=u_1^{-2}u_2^2u_4$,
we have
$\C(x_1,\dots,x_{n+1})^{\tilde{G}}=\C(u_1,\dots,u_{n+1})$. In Step $3$, we obtain \begin{align*}
	p_1=&t_1 x_1 x_2 x_3^2x_{n+2}
	+t_2 x_1^2 x_2 x_3x_{n+2}
	+t_3 x_1 x_4^2 x_{n+2}^2
	+t_4 x_2 x_4^2 x_{n+2}^2
	+l_1x_1x_3x_4x_{n+2}\\
	&+l_2x_2x_3x_4x_{n+2}+l_3x_1^2x_4x_{n+2}+l_4x_1x_2x_4x_{n+2}+x_1x_4h
\end{align*}
and $q_1=x_{n+2}^5.$ Since the degree of $x_2$ in $F_{NC}=p_1$ is $1$, we have that the Noether--Cremona hypersurface $X_{NC}$ is rational, which implies that $X/G$ is rational by Theorem \ref{thm:bir}.
\end{proof}

\begin{remark}
A {\it$\Q$-Fano threefold} is a normal projective threefold $Y$ with $\Q$-factorial terminal singularities, $-K_Y$ ample and Picard number $1$. We obtain the following two new $\Q$-Fano threefolds and their rationality can be determined by Noether--Cremona method.
	\begin{itemize}
		\item[(1)]	Let $X$ be the Fermat cubic threefold (i.e. $X = \{\sum_{i=1}^5 x_i^3 = 0\}$) and $G=\left<[A_1]\right>$ in $\PGL(5,\C)$, where $A_1=\Diag(\xi_3,\xi_3,\xi_3^2,\xi_3^2,1)$. Then the quotient variety $X/G$ turns out to be a $\mathbb{Q}$-Fano threefold with the basket of singularities $\mathcal{B} = \{6 \times \frac{1}{3}(1,2,2)\}$.  By Theorem \ref{thm:Ex3}, $X/G$ is rational. 
		\item[(2)]
		Let $X$ be the cubic threefold  defined by $F=x_1^2x_2+x_2^2x_3+x_3^2x_4+x_4^2x_1+x_5^3$ and $G=\left<[A_1]\right>$ in $\PGL(5,\C)$, where $A_1=\Diag(\xi_5,\xi_5^3,\xi_5^4,\xi_5^2,1)$ and $\xi_5$ is a $5$-th primitive root of unity. Then the quotient variety $X/G$ is $\Q$-Fano with the basket $\mathcal{B}=\left\{4\times \frac{1}{5}(1,2,4)\right\}$. Take rational monomials $u_1=x_2 x_4$, $u_2=x_2x_3^3$, $u_3=x_3^2 x_4$, $u_4=x_1 x_4^2$ and local chart $U=\{x_5=1\}$. Following the steps in Subsection \ref{ss:Gabe}, we have the Noether--Cremona hypersurface $X_{NC}$ of $X/G$ is defined by
		\begin{equation*}
			F_{NC}=x_2^2 x_4^2 x_5+x_1^2 x_2 x_3^2+x_1 x_3^4+x_1 x_3^3 x_4+x_1 x_3^3 x_5.
		\end{equation*}
		Since the degree of $x_5$ in $F_{NC}$ is 1, the quotient variety $X/G$ is rational by Theorem \ref{thm:bir}. 
	\end{itemize}
	These two examples match No. 40245 and  No. 40057 in the {\it Graded Ring Database} (\cite{BK}), respectively. 
	It seems that the rationality of $\mathbb{Q}$-Fano threefolds in these 2 classes was previously unknown and cannot be determined using the methods that are effective for other classes (e.g. \cite{Oka19}, \cite{Pro22}).
\end{remark}

\section{Proof of the main theorem}
In this section, we use the Noether--Cremona method to study  the quotients of smooth cubic threefolds $X$ in a $2$-dimensional family by an order $9$ abelian group $G$ and its order $3$ subgroup $G_1$. We show that $X/G$ are birational to smooth cubic threefolds $X'$ in another $2$-dimensional family (Theorem \ref{thm:C3C3}). Moreover, we prove that the quotients $X/{G_1}$ are rational (Theorem \ref{thm:C3cubic3}). Combining these results, we conclude that $X'$ admit unirational parameterizations of degree $3$ (Theorem \ref{thm:maindetailed}), which implies our main result (Theorem \ref{thm:main}).
\begin{theorem}\label{thm:C3C3}
	Let $X\subset \P^4$ be a smooth cubic threefold defined by
	\begin{equation*}
	F=t_1 x_1^3 + t_2 x_2^3 + t_3 x_3^3 + t_4 x_4^3 + t_5 x_5^3 + t_6 x_1 x_2 x_3 + t_7 x_2 x_4 x_5,
\end{equation*} 
where $t_1,\dots,t_5\in\C^*$, $t_6,t_7\in\C$. Let $G=\left\langle [A_1],[A_2]\right\rangle  \subset {\rm PGL}(5,\C)$, where
		$$A_1=\Diag(\xi_3,1, \xi_3^2, 1, 1) \text{ and } A_2=\Diag(1,\xi_3,\xi_3^2,\xi_3^2,1).$$ 
Then $X/G$ is birational to the smooth cubic threefold defined by the polynomial 
\begin{equation*}
t_1 x_1^2 x_3 + t_2 x_2^3 + t_3 x_1 x_3^2 +t_4 x_4^2 x_5 + t_5 x_4 x_5^2+ t_6 x_1 x_2 x_3 +  t_7 x_2 x_4 x_5.
\end{equation*} In particular, if $X$ is the Fermat cubic threefold, then $X/G$ is birational to $X$.
\end{theorem}

\begin{proof}
We first apply the Noether--Cremona method for $X/G$. Again, we adopt the notation in Subsection \ref{ss:Gabe}. We take the local chart $U=\{x_{5}=1\}$ and choose $\tilde{G}$-invariant rational monomials $$u_1=\frac{x_1 x_3}{x_4}, \,u_2=x_2 x_4,\, u_3=x_3^3 \;{\rm and} \;u_4=x_4^3.$$
By computing monomials in the spaces $\C[x_1,x_2,x_3,x_{4}]_i^{\tilde{G}}$ ($1\le i\le 5$), we conclude that
	$$\C[x_1,x_2,x_3,x_{4}]^{\tilde{G}}=\C[x_2 x_4, x_1 x_2 x_3, x_1^3, x_2^3, x_3^3, x_4^3,  x_1 x_3 x_4^2, x_1^2 x_3^2 x_4]$$
since every monomial in the subring $\bigoplus_{i \geq 6} \C[x_1, x_2,x_3, x_4]_i$ is divisible by at least one of the following monomials: $x_2 x_4$, $x_1 x_2 x_3$, $x_1^3$, $x_2^3$, $x_3^3$, $x_4^3$, $x_1 x_3 x_4^2$. Since
$x_1x_2x_3=u_1u_2$,  $x_1^3=u_1^3u_4u_3^{-1}$, 
$x_2^3=u_2^3u_4^{-1}$,
$x_1 x_3 x_4^2=u_1u_4$ and $x_1^2 x_3^2 x_4=u_1^2u_4,$
we have
$\C(x_1,x_2,x_3,x_4)^{\tilde{G}}=\C(u_1,u_2,u_3,u_4)$. In Step $3$, we obtain \begin{equation}\label{Ex2p1}
	p_1=t_1 x_1^3 x_4^2+t_2 x_2^3 x_3 x_5+t_3 x_3^2 x_4 x_5^2+t_4 x_3 x_4^2 x_5^2+t_5 x_3 x_4 x_5^3+t_6 x_1 x_2 x_3 x_4 x_5+t_7 x_2 x_3 x_4 x_5^2
\end{equation}
and $q_1=x_{5}^5.$
Then the Noether--Cremona hypersurface $X_{NC}\subset \P^{4}$ is a quintic threefold defined by $F_{NC}=p_1$.

Next we apply Noether--Cremona method for the quotient of $X'$ by the trivial group $G'$, where $X':=X_{NC}$, $G':=\{ [A_1']\}$ and $A_1'=I_5$. We choose the affine chart $U'=\{x_1=1\}$ and take
	$$u_2=x_2, u_3=x_3x_5, u_4=x_4, u_5=x_5.$$
	Clearly 
	$\C(x_2,x_3,x_4,x_5)=\C(u_2,u_3,u_4,u_5)$. From \eqref{Ex2p1}, we compute the Noether--Cremona hypersurface $X'_{NC}$ which is defined by \begin{equation}\label{eq:Ex2Fpnc}
		F'_{NC}=t_1 x_1^2 x_4^2 + t_2 x_2^3 x_3 + t_3 x_1 x_3^2 x_4 +t_4 x_3 x_4^2 x_5 + t_5 x_3 x_4 x_5^2+ t_6 x_1 x_2 x_3 x_4 +
		t_7 x_2 x_3 x_4 x_5.
	\end{equation}
	
Finally we apply Noether--Cremona method for the quotient of $X''$ by the trivial group $G'$, where $X'':=X_{NC}'$. We choose the affine chart $U''=\{x_3=1\}$ and take
$$u_1=x_1x_4, u_2=x_2, u_4=x_4, u_5=x_5.$$
From \eqref{eq:Ex2Fpnc}, we get the Noether--Cremona hypersurface $X_{NC}''$ defined by 
\begin{equation*}
		F''_{NC}=t_1 x_1^2 x_3 + t_2 x_2^3 + t_3 x_1 x_3^2 +t_4 x_4^2 x_5 + t_5 x_4 x_5^2+ t_6 x_1 x_2 x_3 +  t_7 x_2 x_4 x_5.
	\end{equation*}
By Theorem \ref{thm:bir}, we have $X/G$ is birational to $X_{NC}''$. The smoothness of $X$ is equivalent to the smoothness of $X_{NC}''$ by a similar argument in the proof of Lemma \ref{lem:smtosm}. The last statement of the theorem is obtained by a suitable linear change of coordinates (cf. \cite[Example 3.3]{YYZ24}).
\end{proof}

\begin{theorem}\label{thm:C3cubic3}
	Let $X\subset \P^4$ be a smooth cubic threefold preserved by $$G_1:=\langle [\Diag(1,\xi_3,\xi_3^2,\xi_3^2,1)]\rangle \subset {\rm PGL}(5,\C).$$ Then  $X/G_1$ is rational.
\end{theorem}

\begin{proof}
Let $A_1':=\Diag(\xi_3,\xi_3,\xi_3^2,\xi_3^2,1)$. Note that $[\Diag(1, \xi_3, \xi_3^2, \xi_3^2, 1)]=[\Diag(\xi_3^2, 1, \xi_3, \xi_3, \xi_3^2)]$. Thus it suffices to show that for any smooth cubic threefold $X'\subset \P^4$ preserved by $G_1':=\langle [A_1']\rangle \subset {\rm PGL}(5,\C)$, the quotient variety $X'/G_1'$ is rational. Let $F'$ be the defining equation of $X'$. By \cite[Lemma 4.8]{WY20}, we have $A_1'(F')=F'$. Note that the space $\C[x_1,\dots,x_5]^{\tilde{G_1}}_3$ of cubic forms invariant by ${\tilde{G_1}}=\left\langle A_1'\right\rangle $ is equal to
	$${\rm span}_\C\{x_1^3, x_1^2 x_2, x_1 x_2^2, x_1 x_3 x_5, x_1 x_4 x_5, x_2^3, x_2 x_3 x_5, x_2 x_4 x_5,
	x_3^3, x_3^2 x_4, x_3 x_4^2, x_4^3, x_5^3\}.$$
Since $X'$ is smooth, by \cite[Proposition 3.3]{OY19}, up to linear change of coordinates, we may assume $F'$ is of the form
	$$F'=t_1x_1^2 x_2+t_2x_1 x_2^2+t_3x_3^2 x_4+t_4x_3 x_4^2+t_5x_1 x_3 x_5+t_6x_2 x_3 x_5+t_7x_1 x_4 x_5+t_8x_2 x_4 x_5+t_9x_5^3,$$
	where $t_1,t_2,t_3,t_4,t_9\in\C^*$, $t_5,t_6,t_7,t_8\in \C$. Then the theorem follows from Theorem \ref{thm:Ex3}.
\end{proof}

\begin{theorem}\label{thm:maindetailed}
	Let $X'\subset \P^4$ be a smooth cubic threefold defined by
	\begin{equation}\label{eq:mainF}
		F'=t_1 x_1^2 x_3 + t_2 x_2^3 + t_3 x_1 x_3^2 +t_4 x_4^2 x_5 + t_5 x_4 x_5^2+ t_6 x_1 x_2 x_3 +  t_7 x_2 x_4 x_5,
	\end{equation} 
	where $t_1,\dots,t_5\in\C^*$, $t_6,t_7\in\C$. 
	Then $X'$ admits a unirational parametrization of degree $3$.
\end{theorem}
\begin{proof}
Let $X, F, G, A_1, A_2$ be as in Theorem \ref{thm:C3C3} and let $G_1:=\langle [A_2]\rangle$. Since $G_1$ is a (normal) subgroup of the abelian group $G$, we have $X/G$ is isomorphic to the quotient of $X/G_1$ by $G/G_1$. Then by Theorems \ref{thm:C3C3} and \ref{thm:C3cubic3}, we conclude the theorem. \end{proof}

\begin{remark}\label{rmk:2dim}
By smoothness of $X'$, we may assume $t_i=1$ ($1\le i\le 5$) by linear change of coordinates. Thus the cubics in Theorem \ref{thm:maindetailed} provide us a $2$-dimensional subvariety in the moduli space of the smooth cubic threefolds.
\end{remark}

Now we are ready to prove Theorems \ref{thm:main} and \ref{thm:main-cor}.
\begin{proof}[Proof of Theorem \ref{thm:main}]
By Theorem \ref{thm:maindetailed} and Remark \ref{rmk:2dim}, there exists a $2$-dimensional family of smooth cubic threefolds admitting a unirational parametrization of degree $3$. Then the theorem follows from \cite[Appendix B]{CG72}.
\end{proof}
\begin{proof}[Proof of Theorem \ref{thm:main-cor}]
If $t_6$ in \eqref{eq:mainF} is zero, then up to linear change of coordinates (cf. \cite[Example 3.3]{YYZ24}), the $F'$ in Theorem \ref{thm:maindetailed} can be written in the form $x_1^3+x_3^3+F_1'(x_2,x_4,x_5)$ by \cite[Lemma 3.9]{YYZ24}. 
	The theorem follows from \cite[Proposition 3.1]{CT17} and Theorem \ref{thm:maindetailed}.
\end{proof}
Note that the smooth cubic hypersurfaces in Theorem \ref{thm:main-cor} are universal ${\rm CH}_{0}$-trivial by Colliot-Th\'{e}l\`{e}ne \cite[Theorem 3.8]{CT17}.
	However, the coprime-degree unirational parametrizations  were previously unknown. 
	Moreover, we have the following observation.

\begin{corollary}
	Let $X$ be a smooth cubic hypersurface as in Theorems \ref{thm:main} and \ref{thm:main-cor}. Then for any integer $m\geq 1$, $X\times \P^m$ admits unirational parametrizations of coprime degrees.
\end{corollary}

	\begin{remark}
		By tracing the steps in our proofs, we have an explicit dominant rational map of degree $3$ from $\P^3$ to the Fermat cubic threefold given by
		\begin{footnotesize}
			$$(h_1 h_2 l_1 l_2^2 l_3^3 +\xi_3 h_1 h_3^3 l_1:
			3 h_1 h_2 h_3 l_1 l_2 l_3:
			-\xi_3 h_1 h_2 l_1 l_2^2 l_3^3 - h_1 h_3^3 l_1:
			h_1^3 h_3 l_3 + \xi_3 h_2 h_3 l_1^3 l_2^2 l_3:
			-\xi_3 h_1^3 h_3 l_3 - h_2 h_3 l_1^3 l_2^2 l_3),$$
		\end{footnotesize}
		where $l_1 = x_1 - \xi_ 3 x_3$, 
		$l_2 = x_2 + x_4$, 
		$l_3 = x_1 - \xi_ 3^2 x_3$, 
		$h_1 = x_1^2 x_3 +x_1x_3^2 + \xi_ 3^2 x_2^2 x_4  -\xi_3 x_2x_4^2$, 
		$h_2 = x_1^2x_2x_3+x_1x_2x_3^2+ x_2^2 x_4^2$, and
		$h_3 = x_1^2 x_3 +x_1x_3^2 + \xi_3x_2^2 x_4 -\xi_ 3^2  x_2x_4^2$.
	\end{remark}

	\begin{remark}
		Clearly, any variety birational to a smooth cubic threefold $X$ in Theorem \ref{thm:main} admits unirational parametrizations of coprime degrees 2 and 3;
		for instance, some smooth Fano threefolds: the blow-up of $X$ in a line, the blow-up of $X$ in a plane cubic curve and the Fano threefolds of degree $14$ and Picard number $1$ associated to $X$.
	\end{remark}
\begin{remark}\label{rmk:posi}
Theorem \ref{thm:main} and its proof are valid over any algebraically closed field $K$ of positive characteristic $\neq2,3$ (see \cite[Appendix B]{CG72}, Remark \ref{rmk:basefield} and proofs of Theorems \ref{thm:C3C3}, \ref{thm:C3cubic3}, \ref{thm:maindetailed}). Thus there is a $2$-dimensional family of smooth cubic threefolds over $K$ with unirational parametrizations of coprime degrees $2$ and $3$. These cubics are irrational (\cite{Mur73}).
\end{remark}

\end{document}